\documentclass[12pt]{article}
\usepackage{textcomp}
\usepackage{amssymb}
\usepackage{latexsym}
\usepackage{amsthm}
\usepackage{amscd}
\usepackage{amsmath}
\usepackage{mathrsfs}
\usepackage{amsfonts}
\usepackage{enumitem}

\usepackage{lineno}
\usepackage{a4wide}
\usepackage{amsmath}
\usepackage{amssymb}
\usepackage{amsthm}
\usepackage{latexsym}
\usepackage{graphicx}
\usepackage[english]{babel}
\usepackage{makeidx}
\usepackage{microtype}

\newtheorem{thm}{Theorem}[section]
\newtheorem{lem}[thm]{Lemma}
\newtheorem{cor}[thm]{Corollary}
\newtheorem{defn-lem}[thm]{Definition-Lemma}
\newtheorem{conj}[thm]{Conjecture}

\newtheorem{prop}[thm]{Proposition}

\theoremstyle{remark}

\theoremstyle{definition}
\newtheorem{defn}[thm]{Definition}
\numberwithin{equation}{section}

\allowdisplaybreaks[4]

\def \CS{{\mathcal S}}

\def \Z{{\mathbb Z}}

\def\map#1.#2.{#1 \longrightarrow #2}
\def\rmap#1.#2.{#1 \dasharrow #2}

\DeclareMathOperator{\depth}{depth}
\DeclareMathOperator{\Sdepth}{Sdepth}
\DeclareMathOperator{\Hdepth}{Hdepth}

\def\fb#1.{\underset #1 \to \times}
\def\pr#1.{\Bbb P^{#1}}
\def\ring#1.{\mathcal O_{#1}}
\def\mlist#1.#2.{{#1}_1,{#1}_2,\dots,{#1}_{#2}}

\def\uloopr#1{\ar@'{@+{[0,0]+(-4,5)} @+{[0,0]+(0,10)}
@+{[0,0]+(4,5)}}^{#1}}

\def\dloopr#1{\ar@'{@+{[0,0]+(-4,-5)} @+{[0,0]+(0,-10)}
@+{[0,0]+(4,-5)}}_{#1}}

\def\rloopd#1{\ar@'{@+{[0,0]+(5,4)} @+{[0,0]+(10,0)}
@+{[0,0]+(5,-4)}}^{#1}}

\newcommand{\rdots}{\mathinner{\mkern1mu\raise1pt\hbox{.}\mkern2mu\raise4pt\hbox{.} \mkern2mu\raise7pt\vbox{\kern7pt\hbox{.}}\mkern1mu}} 

\def\lloopd#1{\ar@'{@+{[0,0]+(-5,4)} @+{[0,0]+(-10,0)}
@+{[0,0]+(-5,-4)}}_{#1}}

\long\def\ignore#1{}
\long\def\ignore#1{#1}

\begin{document}

\begin{center}
{\bf On Two Classes of Closely Related Monomial Ideals}
\end{center}

\begin{center}
Maorong Ge, Jiayuan Lin and Yulan Wang
\end{center}

{\small {\bf Abstract} In \cite{Ge2} we obtained a formula for the Hilbert depth of squarefree Veronese ideals in a standard graded polynomial ring by relating it to the Hilbert depth of powers of the irrelevant maximal ideal. In this paper, we prove that these two Hilbert depth formulas are equivalent to each other. Our result reveals that there is a strong connection between these two classes of seemingly unrelated monomial ideals. We conjecture that their Stanley depths are equivalent as well.}

\section{Introduction}

In recent years, two classes of monomial ideals, the squarefree Veronese ideals and the powers of the irrelevant maximal ideal in a graded polynomial ring, have received special attentions (e.g. \cite{Bir}, \cite{Brun2}-\cite{Keller}). The $\textit{Stanley depths}$ of these two classes of ideals have been conjectured in \cite{Cim1}, \cite{Cim2} and \cite{Keller} and partially confirmed in \cite{Bir}, \cite{Cim1}, \cite{Cim2}, \cite{Ge1} and \cite{Keller}. Their $\textit{Hilbert depths}$ have been obtained in \cite{Brun2} and \cite{Ge2}. At the beginning it seemed that there were no apparent connections between these two classes of ideals. Most results were derived solely on one of them. However, in \cite{Ge2} we computed the Hilbert depth of the squarefree Veronese ideals and found that it could be reduced to the calculation of the Hilbert depth of the powers of the irrelevant maximal ideal. In this paper we further exploit this connection and show that the two formulas for the Hilbert depths of these two classes of ideals are actually equivalent.

Let $K$ be a field and $R=K[x_1,\cdots, x_n]$ be the polynomial ring in $n$ variables. Let $M$ be a finitely generated graded $R$-module. Analogous to the concepts of $\textit{Stanley decomposition}$ and $\textit{Stanley depths}$, W. Bruns et al \cite{Brun1} introduced $\textit{Hilbert decomposition}$ and $\textit{Hilbert depth}$.  For the reader's convenience, let us recall the definition of $\textit{Hilbert decomposition}$ and $\textit{Hilbert}$ $\textit{depth}$ from \cite{Brun1}.

\begin{defn}
Let $M$ be a finitely generated graded $S$-module. A $\textit{Hilbert decomposition}$ of $M$ is a finite family

$$\mathscr{H}=\big (S_i,s_i \big )$$

\noindent such that $s_i \in \Z^m$ (where $m=1$ in the standard graded case and $m=n$ in the multigraded case), $S_i$ is a graded $K$-algebra retract of $R$ for each $i$, and 

$$M \cong \underset{i}{\bigoplus} S_i(-s_i)$$

\noindent as a graded $K$-vector space.

The number $\Hdepth \mathscr{H}=\underset{i}{\min} \depth S_i(-s_i)$ is called the $\textit{Hilbert depth}$ of $\mathscr{H}$. The $\textit{Hilbert depth}$ of $M$ is defined to be

$$\Hdepth M=\max\{\Hdepth \mathscr{H}: \mathscr{H} \hskip .1 cm \textit{is a Hilbert decomposition of M}\}.$$

\end{defn} 

Note that each Stanley decomposition is also a Hilbert decomposition, so Hilbert depth provides an upper bound for the Stanley depth. 

For a finitely generated coarsely (or standard) graded $R$-module $M=\bigoplus_{k \in \Z} M_k$, the $\textit{coarse Hilbert series}$ of $M$ is given by the Laurent series $H_M(T)= \sum_{k \in \Z} H(M,k) T^k$, where $H(M,k)=\dim_K M_k$ is the $\textit{Hilbert function}$ of $M$. We say a Laurent series $\sum_{k \in \Z} a_k T^k$ is $\textit{non-negative}$ if $a_k \ge 0$ for all $k$. It is easy to see that any Hilbert series is non-negative. In \cite{Uli}, J. Uliczka proved that the Hilbert depth of $M$ is equal to $\max \{r:(1-T)^r H_{M}(T) \hskip .1 cm \textit{non-negative}\}$. Using this, W. Bruns et al \cite{Brun2} computed the Hilbert depth of the powers of the irrelevant maximal ideal in $R$. Using the result of \cite{Brun2}, we obtained in \cite{Ge2} that the $\textit{coarse Hilbert series}$ of the squarefree Veronese ideal $I_{n,d}$ is $H_{I_{n,d}}(T)=\overset{n-1}{\underset{i=d-1}{\sum}} \binom{i}{d-1} T^d (1-T)^{-n+i-d+1}$, and its Hilbert depth is equal to $d+\left \lfloor \binom{n}{d+1}/\binom{n}{d} \right\rfloor=d+\left \lfloor \frac{n-d}{d+1} \right\rfloor=d-1+\left \lceil \frac{n-(d-1)}{d+1} \right\rceil$. We remark that the numbers in the formulas for their Hilbert depth of both classes of ideals are exactly those appeared in the corresponding conjectural formulas for their Stanley depths. We expect the following to be true.

\begin{conj}
Let $R=K[x_1,\cdots, x_n]$ be the polynomial ring in $n$ variables. Let $I_{n,d}$ be the squarefree Veronese ideal generated by all squarefree monomials of degree $d$ and let $\mathfrak{m}$ be the irrelevant maximal ideal in $R$. Then the formula $\Sdepth (\mathfrak{m}^s)=
\left \lceil \frac{n}{s+1} \right\rceil$ implies $\Sdepth (I_{n,d})=d-1+\left \lceil \frac{n-(d-1)}{d+1} \right\rceil$ and vice versa, where $s$ is a positive integer.
\end{conj}

In this paper we prove a Hilbert depth version of Conjecture $1.2$.

\begin{thm}
Let $I_{n,d}$ be the squarefree Veronese ideal generated by all squarefree monomials of degree $d$ and let $\mathfrak{m}$ be the irrelevant maximal ideal in $R=K[x_1,\cdots, x_n]$. Then the formulas $\Hdepth (\mathfrak{m}^s)=
\left \lceil \frac{n}{s+1} \right\rceil$ and $\Hdepth (I_{n,d})=d-1+\left \lceil \frac{n-(d-1)}{d+1} \right\rceil$ are equivalent, that is, they imply each other, where $s$ is a positive integer.
\end{thm}

In Proposition $1.5$ we show that Theorem $1.3$ follows easily from the following result.

\begin{thm}
Let $I_{n,d}$ be the squarefree Veronese ideal generated by all squarefree monomials of degree $d$ and let $\mathfrak{m}$ be the irrelevant maximal ideal in $R=K[x_1,\cdots, x_n]$. Let 
$\hat{\mathfrak{m}}=\mathfrak{m} \cap \hat{R}$ be the irrelevant maximal ideal in $\hat{R}=K[x_1,\cdots, x_{n-d+1}]$ and $\hat{\mathfrak{m}}^d$ be the ideal generated by all monomials of degree $d$ in $\hat{R}$. Then $\Hdepth (\hat{\mathfrak{m}}^d) +d-1= \Hdepth (I_{n,d})$.
\end{thm}

\begin{prop}
Theorem $1.4$ implies Theorem $1.3$.
\end{prop}

\begin{proof}

Suppose that $\Hdepth \mathfrak{m}^s=\left \lceil \frac{n}{s+1} \right\rceil$ is true for any positive integer $s$ and the irrelevant maximal ideal $\mathfrak{m}$ in $R=K[x_1,\cdots, x_n]$. Then applying this formula to $\hat{\mathfrak{m}}=\mathfrak{m} \cap \hat{R}$ in $\hat{R}=K[x_1,\cdots, x_{n-d+1}]$ for $s=d$ gives that $\Hdepth (\hat{\mathfrak{m}}^d)=\left \lceil \frac{n-d+1}{d+1} \right\rceil=\left \lceil \frac{n-(d-1)}{d+1} \right\rceil$. By Theorem $1.4$, we have $\Hdepth (I_{n,d})=\Hdepth (\hat{\mathfrak{m}}^d) +d-1= \left \lceil \frac{n-(d-1)}{d+1} \right\rceil+d-1$.

Conversely, suppose that $\Hdepth (I_{n,d})=d-1+\left \lceil \frac{n-(d-1)}{d+1} \right\rceil$ is true for any $n \ge d$. Applying this formula to the squarefree Veronese ideal $I_{n+s-1,s}$ generated by all squarefree monomials of degree $s$ in the polynomial ring $K[x_1, \cdots, x_n, x_{n+1}, \cdots, x_{n+s-1}]$, we have that $\Hdepth (I_{n+s-1,s})=s-1+\left \lceil \frac{n+s-1-(s-1)}{s+1} \right\rceil=s-1+\left \lceil \frac{n}{s+1} \right\rceil$. By Theorem $1.4$, $\Hdepth (\mathfrak{m}^s) +s-1= \Hdepth (I_{n+s-1,s})=s-1+\left \lceil \frac{n}{s+1} \right\rceil$. So $\Hdepth (\mathfrak{m}^s)= \left \lceil \frac{n}{s+1} \right\rceil$. This completes the proof of Proposition $1.5$.
\end{proof}

Now we only need to prove Theorem $1.4$. The proof of Theorem $1.4$ is very simple. We first deduce the fine (or multi-graded) Hilbert series for both squarefree Veronese ideal $I_{n,d}$ (as a $\Z^n$-graded module) in $R=K[x_1, \cdots, x_{n}]$ and $\hat{\mathfrak{m}}^d$ (as a $\Z^{n-d+1}$-graded module) in $\hat{R}=K[x_1,\cdots, x_{n-d+1}]$. Then we show that $H_{I_{n,d}}(T)=(1-T)^{-(d-1)}H_{\hat{\mathfrak{m}}^d} (T)$. Theorem $1.4$ follows immediately by combining this equality with J. Uliczka's result that $\Hdepth (M)=\max \{r:(1-T)^r H_{M}(T) \hskip .1 cm \textit{non-negative}\}$.

\section{Fine and coarse Hilbert series of the squarefree \\ Veronese ideal}

Denote $\CS$ the set of subsets of $\{1, \cdots, n\}$ with cardinality at least $d$. We first deduce the fine Hilbert series of the squarefree Veronese ideal $I_{n,d}$.

\begin{prop}
The fine Hilbert series of the squarefree Veronese ideal $I_{n,d}$ is given by

$$H_{I_{n,d}} (T_1, \cdots, T_n)=\overset{n}{\underset{i=0}{\prod}} (1-T_i)^{-1} \underset{S \in \CS}{\sum} T^{S} (1-T)^{S^c},$$

\noindent where $S^c$ is the complement of $S$ in the set $\{1, \cdots, n\}$, $T^{S}=\underset{i \in S}{\prod} T_i$, and $(1-T)^{S^c}=\underset{j \in S^c}{\prod} (1-T_j)$.
\end{prop}
\begin{proof}

The squarefree Veronese ideal $I_{n,d}$ is generated by all squarefree monomials of degree $d$ in $R=K[x_1, \cdots, x_n]$. A monomial $x_1^{\alpha_1} \cdots  x_n^{\alpha_n}$ is in $I_{n,d}$ if and only if at least $d$ of the $\alpha_i$'s are positive. So the fine Hilbert series of the squarefree Veronese ideal $I_{n,d}$ is given by

$$H_{I_{n,d}} (T_1, \cdots, T_n)=\underset{\text{at least d }\alpha_i>0 } {\sum} T_1^{\alpha_1} \cdots  T_n^{\alpha_n}.$$

By arranging the terms containing the same distinct factors in $H_{I_{n,d}} (T_1, \cdots, T_n)$ into the same group, we have that

$$H_{I_{n,d}} (T_1, \cdots, T_n)=\underset{S \in \CS} {\sum} T^S (\underset{\alpha_i \ge 0} {\sum} \underset{i \in S} {\prod} T_i^{\alpha_i})=\underset{S \in \CS} {\sum} \underset{i \in S} {\prod} \frac{T_i}{1-T_i}=\overset{n}{\underset{i=0}{\prod}} (1-T_i)^{-1} \underset{S \in \CS}{\sum} T^{S} (1-T)^{S^c}.$$
\end{proof}

Let $T_i=T$ in Proposition $2.1$, we immediately have that the coarse Hilbert series of the squarefree Veronese ideal is $H_{I_{n,d}}(T)= \underset{S \in \CS} {\sum} \underset{i \in S} {\prod} \frac{T}{1-T}=\overset{n}{\underset{k=d}{\sum}} \binom{n}{k} \frac{T^k}{(1-T)^k}$. Recall that we have obtained $H_{I_{n,d}}(T)=\overset{n-1}{\underset{i=d-1}{\sum}} \binom{i}{d-1} T^d (1-T)^{-n+i-d+1}$ in \cite{Ge2}. We will show that these two formulas are equal. We need the following lemma.

\begin{lem} For any integer $0 \le i \le n-d$, we have 

$$\binom{i+d-1}{i}=\overset{i}{\underset{l=0}{\sum}} \binom{n}{i-l} (-1)^l \binom{n-d-i+l}{l}.$$
\end{lem}

\begin{proof}

The coefficient of $T^i$ in $(1+T)^n \cdot (1+T)^{-(n-d-i+1)}$ is 
$$\overset{i}{\underset{l=0}{\sum}} \binom{n}{i-l} \binom{-(n-d-i+1)}{l}=\overset{i}{\underset{l=0}{\sum}} \binom{n}{i-l} (-1)^l \binom{n-d-i+l}{l}.$$

Now Lemma $2.2$ follows easily by comparing the coefficients of $T^i$ on both sides of the identity $(1+T)^{i+d-1}=(1+T)^n \cdot (1+T)^{-(n-d-i+1)}$.
\end{proof}

\begin{prop}
$\overset{n}{\underset{k=d}{\sum}} \binom{n}{k} T^k (1-T)^{-k}=\overset{n-1}{\underset{i=d-1}{\sum}} \binom{i}{d-1} T^d (1-T)^{-n+i-d+1}.$
\end{prop}

\begin{proof}
It is sufficient to show that

\begin{equation}
\overset{n}{\underset{k=d}{\sum}} \binom{n}{k} T^k (1-T)^{n-k}=\overset{n-1}{\underset{i=d-1}{\sum}} \binom{i}{d-1} T^d (1-T)^{i-d+1}.
\end{equation}
 
Dividing both sides of Equation $(2.1)$ by $T^d$ and re-numerating indices, we only need to prove   

\begin{equation}
\overset{n-d}{\underset{k=0}{\sum}} \binom{n}{k+d} T^k (1-T)^{n-k-d}=\overset{n-d}{\underset{i=0}{\sum}} \binom{i+d-1}{d-1} (1-T)^{i}.
\end{equation}

Expanding $T^k=[1-(1-T)]^k$ and combining like terms on the left-hand side of Equation $(2.2)$, we have that 

$\overset{n-d}{\underset{k=0}{\sum}} \binom{n}{k+d} T^k (1-T)^{n-k-d}=\overset{n-d}{\underset{k=0}{\sum}} \binom{n}{k+d} \overset{k}{\underset{l=0}{\sum}} (-1)^l \binom{k}{l} (1-T)^l (1-T)^{n-k-d}$

$=\overset{n-d}{\underset{i=0}{\sum}}  \underset{k-l=n-d-i}{\sum} \binom{n}{k+d} (-1)^l \binom{k}{l} (1-T)^{i}=\overset{n-d}{\underset{i=0}{\sum}}  \overset{i}{\underset{l=0}{\sum}} \binom{n}{n+l-i} (-1)^l \binom{n-d-i+l}{l} (1-T)^{i}$

$=\overset{n-d}{\underset{i=0}{\sum}}  \overset{i}{\underset{l=0}{\sum}} \binom{n}{i-l} (-1)^l \binom{n-d-i+l}{l} (1-T)^{i}=\overset{n-d}{\underset{i=0}{\sum}} \binom{i+d-1}{i} (1-T)^{i}=\overset{n-d}{\underset{i=0}{\sum}} \binom{i+d-1}{d-1} (1-T)^{i}$. The last two equalities follows from Lemma $2.2$ and $\binom{i+d-1}{i}=\binom{i+d-1}{d-1}$. This completes the proof of Proposition $2.3$.
\end{proof}

\section{Fine and coarse Hilbert series of the powers of the irrelevant maximal ideal}

Let $\mathfrak{m}$ be the irrelevant maximal ideal in $R=K[x_1, \cdots, x_n]$ and let $\hat{\mathfrak{m}_t}=\mathfrak{m} \cap \hat{R}$ be the irrelevant maximal ideal in $\hat{R}=K[x_1,\cdots, x_{n-t+1}]$. Let $\hat{\mathfrak{m}_t}^s$ be the ideal generated by all monomials of degree $s$ in $\hat{R}$ and $\langle \hat{\mathfrak{m}_t}^s \rangle$ be the ideal generated by $\hat{\mathfrak{m}_t}^s$ in $R$. We will deduce the fine and coarse Hilbert series of $\mathfrak{m}^s$, $\hat{\mathfrak{m}_t}^s$ and $\langle \hat{\mathfrak{m}_t}^s \rangle$. All of those results are well-known. For the reader's convenience, we include these formulas with a short proof here.

\begin{prop}
The fine Hilbert series of $\mathfrak{m}^s$, $\hat{\mathfrak{m}_t}^s$ and $\langle \hat{\mathfrak{m}_t}^s \rangle$ are

$H_{\mathfrak{m}^s} (T_1, \cdots, T_n)=\overset{n}{\underset{i=0}{\prod}} (1-T_i)^{-1}-\overset{s-1}{\underset{k=0}{\sum}} \underset{\underset{i}{\sum} \alpha_i=k}{\sum} T_1^{\alpha_1} \cdots T_n^{\alpha_n},$

$H_{\hat{\mathfrak{m}_t}^s} (T_1, \cdots, T_{n-t+1})=\overset{n-t+1}{\underset{i=0}{\prod}} (1-T_i)^{-1}-\overset{s-1}{\underset{k=0}{\sum}} \underset{\underset{i}{\sum} \alpha_i=k}{\sum} T_1^{\alpha_1} \cdots T_{n-t+1}^{\alpha_{n-t+1}},$ and

$H_{\langle \hat{\mathfrak{m}_t}^s \rangle} (T_1, \cdots, T_{n})=H_{\hat{\mathfrak{m}_t}^s} (T_1, \cdots, T_{n-t+1}) \cdot \overset{n}{\underset{i=n-t+2}{\prod}} (1-T_i)^{-1}$.

\end{prop}
\begin{proof}
A monomial $x_1^{\alpha_1} \cdots  x_n^{\alpha_n}$ is in $\mathfrak{m}^s$ if and only if $\underset{i}{\sum} \alpha_i \ge s$. So the fine Hilbert series of $\mathfrak{m}^s$ is given by

$H_{\mathfrak{m}^s} (T_1, \cdots, T_n)=\underset{\underset{i}{\sum} \alpha_i \ge s}{\sum}T_1^{\alpha_1} \cdots  T_n^{\alpha_n}=\underset{\alpha_i \ge 0}{\sum}T_1^{\alpha_1} \cdots  T_n^{\alpha_n}-\overset{s-1}{\underset{k=0}{\sum}} \underset{\underset{i}{\sum} \alpha_i=k}{\sum} T_1^{\alpha_1} \cdots T_n^{\alpha_n}$

$=\overset{n}{\underset{i=0}{\prod}} (1-T_i)^{-1}-\overset{s-1}{\underset{k=0}{\sum}} \underset{\underset{i}{\sum} \alpha_i=k}{\sum} T_1^{\alpha_1} \cdots T_n^{\alpha_n}$.

$H_{\hat{\mathfrak{m}_t}^s} (T_1, \cdots, T_{n-t+1})$ follows easily from $H_{\mathfrak{m}^s} (T_1, \cdots, T_n)$ by replacing $n$ with $n-t+1$.

Note that a monomial $x_1^{\alpha_1} \cdots  x_n^{\alpha_n}$ is in $\langle \hat{\mathfrak{m}_t}^s \rangle$ if and only if $\overset{n-t+1}{\underset{i=1}{\sum}} \alpha_i \ge s$. So

$H_{\langle \hat{\mathfrak{m}_t}^s \rangle} (T_1, \cdots, T_{n})=\underset{\overset{n-t+1}{\underset{i=1}{\sum}} \alpha_i \ge s}{\sum} T_1^{\alpha_1} \cdots  T_n^{\alpha_n}=\underset{\overset{n-t+1}{\underset{i=1}{\sum}} \alpha_i \ge s}{\sum}T_1^{\alpha_1} \cdots  T_{n-t+1}^{\alpha_{n-t+1}} \cdot \underset{\alpha_i \ge 0}{\sum} \hskip .1 cm \overset{n}{\underset{i=n-t+2}{\prod}} T_{i}^{\alpha_{i}}$

$=H_{\hat{\mathfrak{m}_t}^s} (T_1, \cdots, T_{n-t+1}) \cdot \overset{n}{\underset{i=n-t+2}{\prod}} (1-T_i)^{-1}$.

\end{proof}

Let $T_i=T$ in Proposition $3.1$, we immediately have

\begin{cor}
The coarse Hilbert series of $\mathfrak{m}^s$, $\hat{\mathfrak{m}_t}^s$ and $\langle \hat{\mathfrak{m}_t}^s \rangle$ are

$H_{\mathfrak{m}^s} (T)=(1-T)^{-n}-\overset{s-1}{\underset{k=0}{\sum}} \binom{n+k-1}{k} T^k,$
$H_{\hat{\mathfrak{m_t}^s}} (T)=(1-T)^{-n+t-1}-\overset{s-1}{\underset{k=0}{\sum}} \binom{n-t+k}{k} T^k$ and

$H_{\langle \hat{\mathfrak{m_t}^s} \rangle} (T)=(1-T)^{-n}-(1-T)^{-t+1} \overset{s-1}{\underset{k=0}{\sum}} \binom{n-t+k}{k} T^k$

\end{cor}

Now we are ready to prove Theorem $1.4$.

\section{Proof of Theorem $1.4$}

In order to prove Theorem $1.4$, we need the following lemma.

\begin{lem}
For any non-negative integer $k$, we have

\begin{equation}
\binom{n+k}{k+d}=\overset{n-1}{\underset{i=d-1}{\sum}} \binom{i}{d-1} \binom{n-i+k-1}{k}
\end{equation}
\end{lem}
\begin{proof}
The coefficient of $T^{n-d}$ in the expansion of $(1-T)^{-(d+k+1)}$ is equal to $\binom{d+k+1+n-d-1}{n-d}=\binom{n+k}{n-d}=\binom{n+k}{d+k}$. For any $0 \le i \le n-d$, the coefficient of $T^i$ in $(1-T)^{-d}$ and the coefficient of $T^{n-d-i}$ in $(1-T)^{-(k+1)}$ are $\binom{d+i-1}{i}=\binom{i+d-1}{d-1}$ and $\binom{k+1+n-d-i-1}{n-d-i}=\binom{k+n-d-i}{k}$ respectively. So the coefficient of $T^{n-d}$ in the expansion of $(1-T)^{-(d+k+1)}=(1-T)^{-d} \cdot (1-T)^{-(k+1)}$ is equal to 
$\overset{n-d}{\underset{i=0}{\sum}} \binom{i+d-1}{d-1} \binom{k+n-d-i}{k}=\overset{n-1}{\underset{i=d-1}{\sum}} \binom{i}{d-1} \binom{n-i+k-1}{k}.$ Now Lemma $4.1$ follows easily by comparing the coefficient of $T^{n-d}$ in the expansion of both sides of $(1-T)^{-(d+k+1)}=(1-T)^{-d} \cdot (1-T)^{-(k+1)}$. 
\end{proof}
By J. Uliczka \cite{Uli}, the Hilbert depth of $M$ is equal to $\max \{r:(1-T)^r H_{M}(T) \hskip .05 cm \textit{non-negative}\}$. So to prove Theorem $1.4$, it is sufficient to show that $H_{I_{n,d}}(T)=(1-T)^{-(d-1)}H_{\hat{\mathfrak{m}}^d} (T)$. 
By Corollary $3.2$, $H_{\hat{\mathfrak{m}}^d} (T)=H_{\hat{\mathfrak{m_d}}^d} (T)=(1-T)^{-n+d-1}-\overset{d-1}{\underset{k=0}{\sum}} \binom{n-d+k}{k} T^k$. Hence we only need to show that 
\begin{equation}
\overset{n}{\underset{k=d}{\sum}} \binom{n}{k} T^k (1-T)^{-k}=(1-T)^{-n}-(1-T)^{-(d-1)} \overset{d-1}{\underset{k=0}{\sum}} \binom{n-d+k}{k} T^k
\end{equation}

Multiply both sides of Equation $(4.2)$ by $(1-T)^n$ and using Equation $(2.1)$, we have that

\begin{equation}
\overset{n-1}{\underset{i=d-1}{\sum}} \binom{i}{d-1} T^d (1-T)^{i-d+1}=1-(1-T)^{n-d+1}\overset{d-1}{\underset{k=0}{\sum}} \binom{n-d+k}{k} T^k
\end{equation}

Dividing both sides of Equation $(4.3)$ by $(1-T)^{n-d+1}$ and expanding $(1-T)^{i-n}$ and $(1-T)^{-(n-d+1)}$
respectively, we have

\begin{equation}
\overset{n-1}{\underset{i=d-1}{\sum}} \binom{i}{d-1} T^d \overset{\infty}{\underset{k=0}{\sum}} \binom{n-i+k-1}{k} T^k=\overset{\infty}{\underset{k=d}{\sum}} \binom{n-d+k}{k} T^k
\end{equation}

Dividing Equation $(4.4)$ by $T^d$, we have

\begin{equation}
\overset{\infty}{\underset{k=0}{\sum}}\overset{n-1}{\underset{i=d-1}{\sum}} \binom{i}{d-1} \binom{n-i+k-1}{k} T^k=\overset{\infty}{\underset{k=d}{\sum}} \binom{n-d+k}{k} T^{k-d}=\overset{\infty}{\underset{k=0}{\sum}} \binom{n+k}{k+d} T^{k}
\end{equation}

Equation $(4.5)$ holds true if and only if $\binom{n+k}{k+d}=\overset{n-1}{\underset{i=d-1}{\sum}} \binom{i}{d-1} \binom{n-i+k-1}{k}$. The latter one follows from Lemma $4.1$. This completes the proof of Theorem $1.4$.

{}

{\footnotesize DEPARTMENT OF MATHEMATICS, ANHUI UNIVERSITY, HEFEI, ANHUI, 230039, CHINA}

$\textit{E-mail address:}$ ge1968@126.com

\medskip

{\footnotesize DEPARTMENT OF MATHEMATICS, SUNY CANTON, 34 CORNELL DRIVE, CANTON, 

\hskip .05 cm NY 13617, USA}

$\textit{E-mail address:}$ linj@canton.edu

\medskip

{\footnotesize DEPARTMENT OF MATHEMATICS, SUNY CANTON, 34 CORNELL DRIVE, CANTON, 

\hskip .05 cm NY 13617, USA}

{\footnotesize DEPARTMENT OF MATHEMATICS, ANHUI ECONOMIC MANAGEMENT INSTITUTE, HEFEI, 

\hskip .05 cm ANHUI, 230059, CHINA}

$\textit{E-mail address:}$ wangy@canton.edu
\end{document}